\documentclass[a4paper,twoside,12pt]{article}
\usepackage[a4paper,nohead,includefoot,total={15cm,23.7cm},centering]{geometry}
\usepackage{amsmath}
\usepackage{amssymb}
\usepackage{amsthm}
\usepackage[all]{xy}

\newcommand{\spc}{\ }
\newcommand{\defeq}{\mathrel{\mathop:}=}

\newcommand{\defequiv}{\mathrel{\mathop:}\equiv}
\newcommand{\equivdef}{\equiv\mathrel{\mathop:}}
\newcommand{\ul}[1]{\underline{#1}}

\newcommand{\lh}{\mathop{\mathrm{lh}}}
\newcommand{\concatenation}{{}^\frown}
\newcommand{\Seq}{\text{Seq}}
\newcommand{\cont}{\text{cont}}
\newcommand{\ASNIS}{AS\!N\!I\!S}
\newcommand{\res}{|\!\raisebox{1mm}{$\scriptscriptstyle\setminus$}}

\newcommand{\IPP}{\mathsf{IPP}}
\newcommand{\FIPPzero}{\mathsf{FIPP}_0}
\newcommand{\FIPPone}{\mathsf{FIPP}_1}
\newcommand{\FIPPtwo}{\mathsf{FIPP}_2}
\newcommand{\FIPPtwothree}{\mathsf{FIPP}_{2/3}}
\newcommand{\FIPPthree}{\mathsf{FIPP}_3}
\newcommand{\CUB}{\mathsf{CUB}}
\newcommand{\SigmaZeroCUB}{\Sigma^0_0\text{-}\mathsf{CUB}}
\newcommand{\SigmaOneCUB}{\Sigma^0_1\text{-}\mathsf{CUB}}
\newcommand{\PiCUB}{\Pi^0_1\text{-}\mathsf{CUB}}
\newcommand{\CUBprime}{\mathsf{CUB}'}
\newcommand{\SigmaZeroCUBprime}{\Sigma^0_0\text{-}\mathsf{CUB}'}
\newcommand{\SigmaOneCUBprime}{\Sigma^0_1\text{-}\mathsf{CUB}'}
\newcommand{\PiCUBprime}{\Pi^0_1\text{-}\mathsf{CUB}'}

\newcommand{\BSigmaOne}{B\Sigma^0_1}
\newcommand{\BSigmaTwo}{B\Sigma^0_2}
\newcommand{\BSigmaInfty}{B\Sigma^0_\infty}
\newcommand{\PCM}{\mathsf{PCM}}
\newcommand{\RT}{\mathrm{RT}}
\newcommand{\SigmaOneUB}{\Sigma^0_1\text{-}\mathsf{UB}}
\newcommand{\ExistsUBX}{\exists\text{-}\mathsf{UB}^X}
\newcommand{\PiOneUBRes}{\Pi^0_1\text{-}\mathsf{UB}\res}

\newcommand{\Ztwo}{\mathsf{Z}_2}
\newcommand{\RCA}{\mathsf{RCA}}
\newcommand{\RCAzero}{\mathsf{RCA}_0}
\newcommand{\WKLzero}{\mathsf{WKL}_0}
\newcommand{\ACAzero}{\mathsf{ACA}_0}
\newcommand{\ATRzero}{\mathsf{ATR}_0}
\newcommand{\PiCAzero}{\Pi^1_1\text{-}\mathsf{CA}_0}

\newtheorem{theorem}{Theorem}
\newtheorem{lemma}[theorem]{Lemma}
\newtheorem{proposition}[theorem]{Proposition}
\newtheorem{corollary}[theorem]{Corollary}
\theoremstyle{definition}
\newtheorem{definition}[theorem]{Definition}
\theoremstyle{remark}
\newtheorem{remark}[theorem]{Remark}

\begin{document}

\title{On Tao's ``finitary'' infinite pigeonhole principle\footnote{Published in The Journal of Symbolic Logic, volume 75, number 1, 2010, pages 355--371. \textcopyright~2010, Association for Symbolic Logic.}}
\author{Jaime Gaspar \and Ulrich Kohlenbach\thanks{We are grateful to Terence Tao. The first author was financially supported by the Portuguese Funda\c{c}\~{a}o para a Ci\^{e}ncia e a Tecnologia (grant SFRH/BD/36358/2007). The second author has been supported by the German Science Foundation (DFG Project KO 1737/5-1).}
\medskip\\
Fachbereich Mathematik, Technische Universit\"{a}t Darmstadt\\
Schlossgartenstra\ss{}e 7, 64289 Darmstadt, Germany\\
mail@jaimegaspar.com, kohlenbach@mathematik.tu-darmstadt.de}
\date{28 September 2010}
\maketitle

\begin{abstract}
In 2007, Terence Tao wrote on his blog an essay about soft analysis, hard analysis and the finitization of soft analysis statements into hard analysis statements. One of his main examples was a quasi-finitization of the infinite pigeonhole principle $\IPP$, arriving at the ``finitary'' infinite pigeonhole principle $\FIPPone$. That turned out to not be the proper formulation and so we proposed an alternative version $\FIPPtwo$. Tao himself formulated yet another version $\FIPPthree$ in a revised version of his essay.

We give a counterexample to $\FIPPone$ and discuss for both of the versions $\FIPPtwo$ and $\FIPPthree$ the faithfulness of their respective finitization of $\IPP$ by studying the equivalences $\IPP \leftrightarrow \FIPPtwo$ and $\IPP \leftrightarrow \FIPPthree$ in the context of reverse mathematics. In the process of doing this we also introduce a continuous uniform boundedness principle $\CUB$ as a formalization of Tao's notion of a correspondence principle and study the strength of this principle and various restrictions thereof in terms of reverse mathematics, i.e., in terms of the ``big five'' subsystems of second order arithmetic.
\end{abstract}

\section{Introduction}

In his article \cite{Tao}, T.\ Tao introduced the program of finitizing infinitary principles $P$ in analysis. This is achieved by showing (using compactness and continuity arguments) the existence of a uniform bound on some existential number quantifier (in a suitably reformulated version of $P$, e.g., corresponding to its Herbrand normal form) that is independent from the infinitary input of the principle (typically an infinite sequence in some metric space). From this bound one then reads off that the new (``finitary'') principle actually only refers to some finite part (e.g., a finite initial segment in the case of a sequence) of that infinitary input. As two of his prime examples he discusses the convergence principle for bounded monotone sequences of reals ($\PCM$) and the infinitary pigeonhole principle ($\IPP$). As observed in \cite{KohlenbachApplied}, the finitary version of $\PCM$ proposed by Tao directly follows from a well-studied proof-theoretic construction due to the second author, the so-called monotone G\"odel functional interpretation of $\PCM$. In \cite{KohlenbachApplied} a similar case is made concerning $\IPP$, i.e., it is shown that the monotone functional interpretation of $\IPP$ leads to a ``finitary'' version $\FIPPzero$ similar, but not identical, to the one proposed by Tao in his first 2007 posting of \cite{Tao} ($\FIPPone$). Like Tao, we use the prefix ``finitary'' here in quotation marks as neither of the finitizations of $\IPP$ is strictly finitary (in the sense the finitary form of $\PCM$ is) since non-finitary (in fact 2nd order) conditions on the Herbrand index function need to be imposed.

One difference between $\FIPPzero$ and $\FIPPone$ is that the former is formulated in a language of primitive recursive functionals whereas the latter is formulated in terms of sets and finitary set-functions. In closing the gap between the two formulations the second author reformulated $\FIPPzero$ into a variant $\FIPPtwo$ in the same vocabulary as the latter. However, as it turns out, $\FIPPtwo$ has a slightly weaker conclusion than $\FIPPone$. Subsequently, the first author found a counterexample to $\FIPPone$ (see section  \ref{section-counterexample} below). In reaction to that counterexample, Tao modified $\FIPPone$ (in a revised posting of \cite{Tao} from August 2008) to yet another version $\FIPPthree$ which keeps the original conclusion of $\FIPPone$ but strengthens the premise of the latter principle. In order to compare the two finitizations $\FIPPtwo$ and $\FIPPthree$ w.r.t.\ their faithfulness as finitizations of $\IPP$ we investigate in this paper the strength of the equivalences $\IPP\leftrightarrow \FIPPtwo$ and $\IPP \leftrightarrow \FIPPthree$ in terms of the systems $\RCAzero$, $\WKLzero$ and $\ACAzero$ from the program of reverse mathematics (see \cite{Simpson}). For $\FIPPzero$ it follows from the reasoning given in \cite{KohlenbachApplied} that it implies $\IPP$ over a system of functionals of finite type that is conservative over Kalmar elementary arithmetic and that the implication $\IPP\to\FIPPzero$ follows with an additional use of WKL (needed to show that continuous functionals $\Phi:2^{\mathbb{N}}\to\mathbb{N}$ are bounded, see \cite{KohlenbachFoundation,Simpson}). This suggests that the version $\FIPPtwo$ that was prompted by $\FIPPzero$ has a similar behavior: more precisely we show that $\RCAzero$ proves $\FIPPtwo\to \IPP$ while $\WKLzero$ proves $\IPP\to \FIPPtwo$.

For $\FIPPthree$ the direction $\FIPPthree\to\IPP$ still follows in $\RCAzero$. The implication $\IPP\to \FIPPthree$ can be established by an application of the Bolzano-Weierstra\ss{} property of the compact metric space $[n]^{\mathbb{N}}$ (with respect to the Baire metric) which in turn is provable in (and in fact equivalent to) $\ACAzero$ (see \cite{Simpson}). So $\ACAzero$ proves $\IPP\rightarrow \FIPPthree$. This, however, is unsatisfactory as $\ACAzero$ is much stronger than $\IPP$ itself, whereas $\WKLzero$ does not prove $\IPP$ by a result due to \cite{Hirst}. So it is natural to try to establish the implication $\IPP\to \FIPPthree$ by a WKL-type ``Heine-Borel''-compactness argument rather than by using sequential compactness (requiring $\ACAzero$). Towards this goal and aiming at a formalization of Tao's informal notion of ``correspondence principle'' from \cite{Tao(08)} we formulate a ``continuous uniform boundedness principle'' $\CUB$ that generalizes the usual FAN-uniform boundedness obtained from (the contrapositive form of) WKL. In fact, $\CUB$ restricted to $\Sigma^0_1$ formulas, denoted by $\SigmaOneCUB$, is equivalent to $\WKLzero$ over $\RCAzero$ and the proof that $\WKLzero$  implies $\IPP\to\FIPPtwo$ can nicely be recasted as an application of $\SigmaOneCUB$ as we will do below. Also $\IPP\to\FIPPthree$ can be established by an application of $\CUB$. However, this time it seems that a $\Pi^0_1$ instance, i.e., a use of $\PiCUB$ is needed. Unfortunately, $\PiCUB$ is no longer derivable in $\WKLzero$ but, in fact, is equivalent to $\ACAzero$ (over $\RCAzero$) which shows that the strength of the correspondence principle as formalized by $\CUB$ crucially depends on the logical complexity of the instance involved. In fact, over $\RCA$ it turns out that the unrestricted $\CUB$ even is equivalent to full second order comprehension over numbers, i.e., to $\Ztwo$. While leaving open the question whether $\WKLzero$ proves $\IPP \to \FIPPthree$, the results in this paper may suggest that the answer is negative and at the same time show that the logical structure of the formula to which a correspondence principle such as $\CUB$ is applied matters in determining how close a finitization of some infinitary principle stays to that principle.

The following diagram summarizes the picture established in this paper:
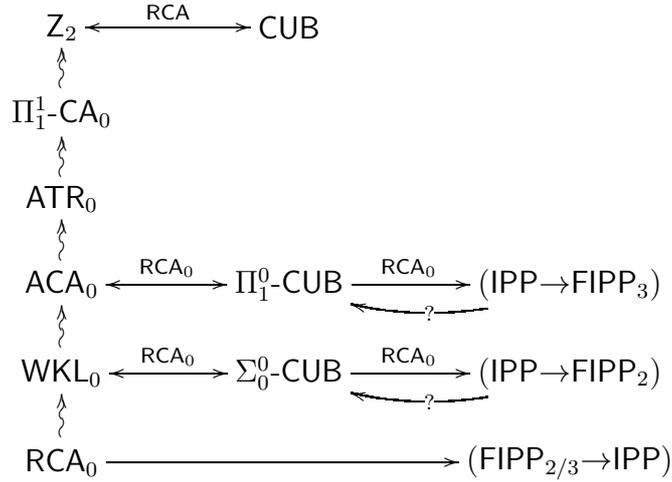
\begin{equation*}
  \xymatrix@C=40pt@R=15pt{
  \Ztwo \ar@{<->}[r]^{\RCA\ \, }                   & \CUB                                   &                                               \\
  \PiCAzero \ar@{~>}[u]                            &                                        &                                               \\
  \ATRzero \ar@{~>}[u]                             &                                        &                                               \\
  \ACAzero \ar@{~>}[u] \ar@{<->}[r]^{\RCAzero\ \;} & \PiCUB \ar[r]^{\RCAzero\quad\ }        & \ar@/^0.9pc/[l]|? (\IPP {\to} \FIPPthree)     \\
  \WKLzero \ar@{~>}[u] \ar@{<->}[r]^{\RCAzero\ \,} & \SigmaZeroCUB \ar[r]^{\RCAzero\quad\ } & \ar@/^0.9pc/[l]|? (\IPP {\to} \FIPPtwo)       \\
  \RCAzero \ar@{~>}[u] \ar@{->}[rr]                &                                        & (\FIPPtwothree {\to} \IPP)
  }
\end{equation*}

\section{Definitions}

In this paragraph we collect some notation and formulate the infinite pigeonhole principle $\IPP$ as well as the three ``finitary'' infinite pigeonhole principles $\FIPPone$, $\FIPPtwo$ and $\FIPPthree$.

All the definitions take place in the context of the language of $\RCAzero$, where $\RCAzero$ is the base system used in reverse mathematics (see \cite{Simpson} for details). All undefined notations are to be understood in the sense of \cite{Simpson}. We need to be rather formal in our definitions working over the weak base system $\RCAzero$. For example, in point 1 of definition \ref{def:asymptoticallyStable} we need to assume the existence of the union of an infinite sequence of sets, since $\RCAzero$ in general doesn't prove that such a union exists. But over sufficiently strong systems such as $\ACAzero$, the clause stating the existence of the union set is redundant.

Tao formulated his ``finitary'' infinite pigeonhole principle using \emph{set functions}, i.e., functions that take as input a finite subset of $\mathbb{N}$ and return as output a natural number. Those are, however, objects of a higher type than those available in the language of $\RCAzero$, so we had to reformulate Tao's principle using functions from $\mathbb{N}$ to $\mathbb{N}$ by identifying a finite subset of $\mathbb{N}$ with a natural number encoding it.

\begin{definition}
  We denote by $[i]$ the set $\{j : j \leq i\}$ of the first $i + 1$ natural numbers. If $i = 0$, then we make the convention that $[i - 1]$ is the
empty set $\emptyset$.
\end{definition}

\begin{definition}
  If $l \in \Seq$, then we define $A_l$ to be the set \emph{encoded} by the finite sequence with code $l$, i.e., $A_l \defeq \{l(i) : i < \lh l\}$. We say that $l$ is a \emph{code} of a set $A$ if $A = A_l$. One can also consider the minimal code which then is called \emph{the} code of $A$.
\end{definition}

\begin{definition}
  If $f : \mathbb{N} \to \mathbb{N}$ is a function and $m > 0$ then we define $\bar f m$ to be the code of the finite sequence $\langle f(0),\ldots,f(m - 1) \rangle$. For $m = 0$ we make the convention that $\bar f m$ is the code of the empty sequence $\langle \rangle$. If $s \in \Seq$, then we denote the function that extends $s$ by zeros by $s \concatenation o$.
\end{definition}

\begin{definition}
  Let $f : X \to Y$ be a function between sets $X,Y$. We define $|A| = m$ to mean ``exists an $f : [m - 1] \to A$ one-to-one and onto''. Then we define $|A| \geq m \defequiv \exists m' \spc (|A| = m' \wedge m' \geq m)$, and analogously for $|A| > m$, $|A| \leq m$ and $|A| < m$.
\end{definition}

\begin{definition} \mbox{}
\label{def:asymptoticallyStable}
  \begin{enumerate}
    \item A sequence $(l_m)$ represents a \emph{nested sequence with union} of finite subsets of $\mathbb{N}$ if and only if $\forall m \spc (l_m \in \Seq)$, $\forall m \spc (A_{l_m} \subseteq A_{l_{m + 1}})$ and $\bigcup_m A_{l_m}$ exists.
    \item A sequence $(l_m)$ \emph{weakly converges} to an infinite set $A$ if and only if $\forall m \spc (l_m \in \Seq)$ and for all finite sets $B$ we have $\exists i \spc \forall j \geq i \spc (A_{l_j} \cap B = A \cap B)$. Then we say that $(l_m)$ \emph{weakly converges} if and only if it weakly converges to some infinite set.
    \item A function $F : \mathbb{N} \to \mathbb{N}$ is \emph{extensional} if and only if $\forall l,l' \in \Seq \spc [A_l = A_{l'} \to F(l) = F(l')]$. Alternatively, one can always use the unique minimal code which allows one to drop the extensionality requirement.
    \item A function $F : \mathbb{N} \to \mathbb{N}$ is \emph{asymptotically stable}, denoted by $F \in AS$, if and only if it is extensional and for all nested sequences with union $(l_m)$ we have $\exists i \spc \forall j \geq i \spc [F(l_i) = F(l_j)]$.
    \item A function $F : \mathbb{N} \to \mathbb{N}$ is \emph{asymptotically stable near infinite sets}, denoted by $F \in \ASNIS$, if and only if it is extensional and for all weakly convergent sequences $(l_m)$ we have $\exists i \spc \forall j \geq i \spc [F(l_i) = F(l_j)]$.
  \end{enumerate}
\end{definition}

\begin{remark}
  A nested sequence with finite union is never weakly convergent (otherwise it would converge to the finite union but a weakly convergent sequence must converge to an infinite set). Every nested sequence with infinite union is weakly convergent (to the infinite union), but there are weakly convergent sequences that are not nested (e.g., $A_{l_m} \defeq [m] \cup \{m + 2\}$ weakly converges to $\mathbb{N}$ but is not a chain).

  We have $\ASNIS \subsetneq AS$ (if $F \in \ASNIS$, then $F$ stabilizes over a nested chain with finite union because the chain itself stabilizes and $F$ also stabilizes over a nested chain with infinite union because such a chain is weakly convergent, so $F \in AS$; the $F \in AS$ from the counterexample to $\FIPPone$ below is not in $\ASNIS$ otherwise it would also be a counterexample to true $\FIPPthree$).
\end{remark}

\begin{definition} \mbox{}
  \begin{enumerate}
    \item The \emph{infinite pigeonhole principle} $\IPP$ is the principle: every coloring $f$ of $\mathbb{N}$ into $n + 1$ colors has an infinite color class $f^{-1}(c)$. In symbols:
      \begin{equation*}
        \forall n \spc \forall f : \mathbb{N} \to [n] \spc \exists c \in [n] \spc [f^{-1}(c) \text{ infinite}],
      \end{equation*}
      where the set $f^{-1}(c)$ exists by $\Sigma^0_0$ comprehension: $\forall x \spc [x \in f^{-1}(c) \leftrightarrow (x,c)\in f]$.
    \item The \emph{first ``finitary'' infinite pigeonhole principle} $\FIPPone$ is the principle: for all asymptotically stable functions $F$ there exists a $k$ such that every coloring $f$ of $[k]$ into $n + 1$ colors has a color class $A = f^{-1}(c)$ that is ``big'' in the sense of $|A| > F(A)$. In symbols:
      \begin{equation*}
        \forall n \spc \forall F \in AS \spc \exists k \spc \forall f : [k] \to [n] \spc \exists l \in \Seq \spc \exists c \in [n] \spc [A_l = f^{-1}(c) \wedge |A_l| > F(l)].
      \end{equation*}
    \item The \emph{second ``finitary'' infinite pigeonhole principle} $\FIPPtwo$ is the principle: for all asymptotically stable functions $F$ there exists a $k$ such that every coloring $f$ of $[k]$ into $n + 1$ colors has a monochromatic set $A$ that is ``big'' in the sense of $|A| > F(A)$. In symbols:
      \begin{equation*}
        \forall n \spc \forall F \in AS \spc \exists k \spc \forall f : [k] \to [n] \spc \exists l \in \Seq \spc \big[A_l \subseteq [k] \wedge |A_l| > F(l) \wedge f|_{A_l} \text{ constant}\big].
      \end{equation*}
    \item The \emph{third ``finitary'' infinite pigeonhole principle} $\FIPPthree$ is analogous to $\FIPPone$ but with $AS$ replaced by $\ASNIS$.
  \end{enumerate}
\end{definition}

\begin{remark}
  $\IPP$ can also be formulated without reference to the set $f^{-1}(c)$ as 
  \begin{equation*}
    \forall n \spc \forall f : \mathbb{N} \to [n] \spc \exists c \in [n] \spc [\forall m \spc \exists k > m \spc (f(k) = c)].
  \end{equation*}
  In the presence of $\Sigma^0_0$ comprehension (and hence over $\RCAzero$) there is no difference between the two formulations.
\end{remark}

By a well-known result due to J.\ L.\ Hirst \cite{Hirst}, $\IPP$ is equivalent to the bounded collection principle for $\Sigma^0_2$ formulas (often called $\BSigmaTwo$, though set parameters are allowed in the context of $\RCAzero$) and is not provable in $\WKLzero$:

\begin{proposition}[\cite{Hirst}]
  \label{Hirst}
  $\WKLzero$ does not prove $\IPP$.
\end{proposition}

Since $\BSigmaTwo$ and hence $\IPP$ easily follows from $\Sigma^0_2$ induction we have that $\RCA$ proves $\IPP$ as well as $\ACAzero$ proves $\IPP$.

\section{Technical lemmas}

In this section we start by collecting in lemma \ref{lemma:cardinality} some folklore properties about the cardinality of finite sets that we will need later. We first note that the formulas $\lh(s) = m$, $s(i) = m$, $s\subseteq t$ (expressing that the finite sequence encoded by $s$ is an initial segment of the sequence encoded by $t$), etc.\ are all $\Sigma^0_0$ (see \cite{Simpson}).

At some point we will need to talk about (continuous) functionals $\phi : [n]^\mathbb{N} \to \mathbb{N}$ within $\RCAzero$, and to do so we need to show the existence of a code (in the sense of \cite{Simpson}) for them. As shown in \cite{KohlenbachFoundation}, the existence of such a code is equivalent to the existence of a so-called associate of $\phi$ in the sense of Kleene and Kreisel. In the cases at hand it turns out to be easier to construct an associate rather than to produce a code directly. For completeness we include lemma \ref{lemma:associateImpliesHavingCode} which shows that the existence of an associate implies the existence of a code. In the first point of lemma \ref{lemma:uniformContinuity} we show that every $\Sigma^0_0$ formula is provably ``uniformly continuous'' in $\RCAzero$. In the second point we prove that every formula of the form $\forall f : \mathbb{N} \to [n] \spc A(f)$ with $A\in \Sigma^0_0$ is (over $\RCAzero$) equivalent to a $\Pi^0_1$ formula. Finally, in the first point of lemma \ref{lemma:limitUniquenessAndStabilityPoint} we show that if $F \in \ASNIS$ and $A$ is an infinite set, then the stable value that $F$ eventually attains on a sequence weakly converging to $A$ doesn't depend on the sequence. In the second point we show that if $F \in \ASNIS$, then $F$ is ``continuous'' in the sense of the Baire space $\mathbb{N}^\mathbb{N}$ with the metric
\begin{equation*}
  d(f,g) \defeq
  \begin{cases}
    2^{-m} & \text{if exists } m = \min m' \spc [f(m') \neq g(m')], \\
    0      & \text{otherwise,}
  \end{cases}
\end{equation*}
at points that are characteristic functions of infinite sets $A$.

\begin{lemma}
  \label{lemma:cardinality}
  $\RCAzero$ proves the following.
  \begin{enumerate}
    \item $A$ is a finite if and only if $\exists m \spc (|A| = m)$ if and only if $A$ has a code.
    \item If $l \in Seq$, then the formulas $|A_l| = m$, $|A_l| <m$, $|A_l| > m$ are equivalent to $\Sigma^0_0$ formulas.
    \item If $A$ is infinite, then $\forall i \spc \exists j \spc (|A \cap [j - 1]| = i)$.
    \item If $A$ and $B$ are finite sets and $A \subseteq B$, then $|A| \leq |B|$.
  \end{enumerate}
\end{lemma}

\begin{lemma}[\cite{KohlenbachFoundation}]
  \label{lemma:associateImpliesHavingCode}
  If $\alpha : \mathbb{N} \to \mathbb{N}$ is an associate of $\phi : [n]^\mathbb{N} \to \mathbb{N}$, i.e.,
  \begin{enumerate}
    \item $\forall \beta : \mathbb{N} \to [n] \spc \exists m \spc [\alpha(\bar\beta m) > 0]$;
    \item $\forall \beta : \mathbb{N} \to [n] \spc \forall m \spc \big[\big(m = \min m' \spc \alpha(\bar \beta m') > 0\big) \to \alpha(\bar \beta m) = \phi(\beta) + 1\big]$;
  \end{enumerate}
  then $\phi$ has a code as a continuous function $[n]^\mathbb{N} \to \mathbb{N}$ in the sense of the definition II.6.1 in \cite{Simpson}.
\end{lemma}

\begin{proof}
  We may assume that $\alpha$ is a \emph{neighborhood function}, i.e., $\forall i,j \spc [i \subseteq j \wedge \alpha(i) > 0 \to \alpha(i) = \alpha(j)]$ for otherwise we would replace $\alpha$ by the associate of $\phi$ and neighborhood function
  \begin{equation*}
    \alpha'(j) \defeq
    \begin{cases}
      \alpha(i) & \text{if $j \in \Seq$ and exists the shortest $i \subseteq j$ such that $\alpha(i) > 0$,} \\
      0         & \text{otherwise.}
    \end{cases}
  \end{equation*}

  In \cite{Simpson} (example II.5.5) a construction is given of a code $A,d$ (where $A \subseteq \mathbb{N}$ and $d : A \times A \to \mathbb{R}$) for infinite product spaces of complete separable metric spaces. Particularizing the construction for $[n]^\mathbb{N} = \prod_{i = 0}^\infty [n]$ we get
  \begin{gather*}
    A = \{\langle a_i : i \leq j \rangle : j \in \mathbb{N} \wedge \forall i \leq j \spc (a_i \in [n])\}, \\
    d(a,b) = \sum_{i = 0}^\infty 2^{-i} \cdot \frac{|(a \concatenation o)(i) - (b \concatenation o)(i)|}{1 + |(a \concatenation o)(i) - (b \concatenation o)(i)|}.
  \end{gather*}
  One easily verifies that
  \begin{enumerate}
    \item $\forall a,b \in A \spc [d(a,b) < 2^{-r} \to \overline{a \concatenation o} \, r = \overline{b \concatenation o} \, r]$;
    \item $\forall a,b \in A \spc [\overline{a \concatenation o} \, r = \overline{b \concatenation o} \, r \to d(a,b) < 2^{-(r - 1)}]$.
  \end{enumerate}

  Let $B(a,r,b,s)$ be a $\Sigma^0_0$ formula expressing
  \begin{equation*}
    2^{-(n + 1)} < r \leq 2^{-n} \wedge \alpha(\overline{a \concatenation o} \, n) > 0 \wedge |\alpha(\overline{a \concatenation o} \, n) - 1 -b| < s.
  \end{equation*}
  Let $\Phi \subseteq \mathbb{N} \times A \times \mathbb{Q}^+ \times \mathbb{N} \times \mathbb{Q}^+$ be defined by $\Sigma^0_0$ comprehension by $(n,a,r,b,s) \in \Phi \leftrightarrow B(n,a,r,b,s)$ and define $(a,r) \Phi (b,s) \defequiv \exists n \spc [(n,a,r,b,s) \in \Phi]$.

  It is straightforward (though tedious) to verify that $\Phi$ is indeed a code for the continuous function $\phi:[n]^{\mathbb{N}}\to\mathbb{N}$.
\end{proof}

\begin{remark}
  The lemma, as stated, doesn't fit the language of $\RCAzero$ since it refers to the third order object $\phi$. However, under the following interpretation it is provable in $\RCAzero$: if $\alpha$ is an associate (i.e., it satisfies condition 1 of the lemma) and $\Phi$ is the code presented in the proof of the lemma, then for all $f : \mathbb{N} \to [n]$ the value on $f$ extracted from $\alpha$ is equal to the value on $f$ extracted from $\Phi$.
\end{remark}

\begin{lemma}
  \label{lemma:uniformContinuity}
  Let $A(f)$ be a $\Sigma^0_0$ formula, $f$ be a set variable and $\ul{x}$ be a tuple of distinguished number variables in $A(f)$.
  \begin{enumerate}
    \item $\RCAzero$ proves $\forall z \spc \exists y \spc \forall f,g : \mathbb{N} \to [n] \spc \big[\bar f y = \bar g y \to \forall \ul{x} \leq z \spc \big(A(f) \leftrightarrow A(g)\big)\big]$.
    \item There exists a $\Sigma^0_0$ formula $B(m)$ such that $\RCAzero$ proves $\forall f : \mathbb{N} \to [n] \spc A(f) \leftrightarrow \forall m \spc B(m)$. In particular, $\forall f : \mathbb{N} \to [n] \spc A(f)$ is equivalent to a $\Pi^0_1$ formula.
    \item There exists a $\Sigma^0_0$ formula $C(m)$ such that $\RCAzero$ proves $\forall f : \mathbb{N} \to [n] \spc [A(f) \leftrightarrow \forall m \spc C(\bar f m)]$.
  \end{enumerate}
\end{lemma}

\begin{proof}
  1.~The proof is by induction on the structure of formulas. If $f$ doesn't occur in an atomic formula $A$, then the result is obvious. If it occurs, then $A$ must be of the form $t(\ul{x}) \in f$, i.e., (abbreviating $t(\ul{x})$ by $t$) $\exists i,j \leq t \spc [t = (i,j) \wedge f(i) = j]$. We prove by induction on the structure of the number term $t(\ge i)$ that $\exists w \spc \forall \ul{x} \leq z \spc (t \leq w)$ (for example, if $t \equiv t_1 \cdot t_2$ and by induction hypothesis we have $\exists w_1 \spc \forall \ul{x} \leq z \spc (t_1 \leq w_1)$ and $\exists w_2 \spc \forall \ul{x} \leq z \spc (t_2 \leq w_2)$, then $w \defeq w_1 \cdot w_2$ is such that $\forall \ul{x} \leq z \spc (t \leq w)$). Then $y \defeq w + 1$ works.

  For the negation $\neg A$ of $A$ we take the same $y$ that by induction hypothesis works for $A$. For conjunction $A \wedge B$ we take the maximum of the $y$'s working for $A$ and $B$, and analogously for disjunction, implication and equivalence. For the bounded universal quantifier $\forall i < t \spc A(i)$, by induction hypothesis we have $\forall z \spc \exists y \spc \forall f,g : \mathbb{N} \to [n] \spc \big[\bar fy = \bar g y \to \forall \ul{x},i \leq z \spc \big(A(f,i) \leftrightarrow A(g,i)\big)\big]$. Thus taking $z' = \max (z,t)$ we get an $y$ such that for all $f,g : \mathbb{N} \to [n]$, if $\bar f y = \bar g y$, then for all $\ul{x}\leq z$ we have $\forall i<t \spc [A(f,i) \leftrightarrow A(g,i)]$. Hence $\forall i < t \spc A(f,i) \leftrightarrow \forall i < t \spc A(g,i)$. Argue analogously for the bounded existential quantifier.

  2.~Each occurrence of $f$ in $A(f)$ must be in the form $t \in f$. Let $B(m)$ be the formula obtained from $A(f)$ by replacing each such occurrence $t \in f$ by the $\Sigma^0_0$ formula
  \begin{equation*}
    \begin{split} 
      C(m,t,n) \defequiv m \in \Seq \wedge \exists i,j \leq t \hspace*{7.3cm} \\
      \Big[t = (i,j) \wedge \Big(i < \lh m \to j = \min\big(n,m(i)\big)\Big) \wedge (i \geq \lh m \to j = 0)\Big],
    \end{split}
  \end{equation*}
  where $[ \ldots ]$ expresses that $j =  \min\big(n,(m \concatenation o)(i)\big)$. Then $B(m)$ is a $\Sigma^0_0$ formula. Let us prove $\forall f : \mathbb{N} \to [n] \spc A(f) \leftrightarrow \forall m \spc B(m)$.

  `$\rightarrow$':~Take any $m \in \Seq$. We define $f : \mathbb{N} \to [n]$ by $f(i) \defeq \min\big(n,m(i)\big)$ if $i < \lh m$ and $f(i) \defeq 0$ if $i \geq \lh m$. Then $t \in f \leftrightarrow C(m,t,n)$, so $A(f) \leftrightarrow B(m)$.

  `$\leftarrow$':~Take any $f : \mathbb{N} \to [n]$. By point 1, let $y$ be such that $\forall g : \mathbb{N} \to [n] \spc \big[\bar f y = \bar g y \to \big(A(f) \leftrightarrow A(g)\big)\big]$. In particular, taking $g = \bar f y \concatenation o$ we get $A(f) \leftrightarrow A(\bar f y \concatenation o)$. Let $m \defeq \bar f y$. Then $t \in \bar f y \concatenation o \leftrightarrow C(m,t,n)$, so $B(m) \leftrightarrow A(\bar f y \concatenation o) \leftrightarrow A(f)$.

  3.~First we easily show, by induction on the structure of the term $t(i)$, that $\RCAzero$ proves $i \leq q \to t(i) \leq t(q)$.

  Let $A'(a)$ be the formula obtained from $A(f)$ by (adding the assumption $a\in Seq$ and) replacing each instance of $q \in f$ by $\exists i,j \leq q \spc [q = (i,j) \wedge a(i) = j]$. We show, by induction on the structure of $A(f)$, that there exists a term $t$ such that $\RCAzero$ proves $\forall f : \mathbb{N} \to [n] \spc \big[m \geq t \to \big(A(f) \leftrightarrow A'(\bar f m)\big)\big]$. For an atomic formula $A$, if $f$ doesn't occur, then the result is obvious; if $f$ occurs in $A$, then $A$ is of the form $q \in f$, that is equivalent to $\exists i,j \leq q \spc [q = (i,j) \wedge f(i) = j]$, so $t \defeq q + 1$ works. For $\neg A$ we take the same $t$ that works for $A$. For $A \wedge B$, $A \vee B$, $A \to B$ and $A \leftrightarrow B$ we take the sum of the $t$'s working for $A$ and $B$. For $\forall i < q \spc A(i)$ we have by induction hypothesis a term $t'(i)$ that works for $A(i)$, so using the previous paragraph we see that the term $t(q)$ works for $\forall i < q \spc A(i)$. Argue, analogously for $\exists i < q \spc A(i)$.

  Finally, using the previous paragraph we easily see that $C(a) \defequiv m \geq t \to A'(a)$ works.
\end{proof}

\begin{lemma}
  \label{lemma:limitUniquenessAndStabilityPoint}
  $\RCAzero$ proves the following.
  \begin{enumerate}
    \item For all $F \in \ASNIS$ and for all infinite sets $A$, there exists a $c$ such that for all sequences $(l_m)$ weakly convergent to $A$, we have 
$\exists i \spc \forall j \geq i \spc [F(l_j) = c]$.
    \item For all $F \in \ASNIS$ and for all infinite sets $A$, there exist $c$ and $d$ such that $\forall l \in \Seq \spc [A_l \cap [d] = A \cap [d] \to F(l) = c]$.
  \end{enumerate}
\end{lemma}

\begin{proof}
  1.~Let $F \in \ASNIS$ and let $A$ be an infinite set. First we define $c$. We define by primitive recursion a sequence $(l_m)$ where each $l_m$ is such that $A_{l_m} = A \cap [m]$. Clearly $(l_m)$ weakly converges to $A$, so since $F \in \ASNIS$ we have $\exists h \spc \forall j \geq h \spc [F(l_h) 
= F(l_j)]$. Let $c \defeq F(l_h)$.

  Consider an arbitrary sequence $(l'_m)$ weakly converging to $A$. Let us define $i$. We have $\exists h' \spc \forall j \geq h' \spc [F(l'_{h'}) = F(l'_j)]$. By primitive recursion define the sequence $(l''_m)$ by
  \begin{equation*}
    l''_m \defeq
    \begin{cases}
      l_{m/2}         & \text{if $m$ is even,} \\
      l'_{(m - 1)/2}  & \text{if $m$ is odd.}
    \end{cases}
  \end{equation*}
  Since both $(l_m)$ and $(l'_m)$ weakly converge to $A$, then also $(l''_m)$ weakly converges to $A$, and so since $F \in \ASNIS$ we have $\exists h'' \spc \forall j \geq h'' \spc [F(l''_{h''}) = F(l''_j)]$. Let $i \defeq \max (2h,2h',h'')$.

  It remains to prove $\forall j \geq i \spc [F(l'_j) = c]$. Since $i \geq h'$, and so $\forall j \geq i \spc [F(l'_{h'}) = F(l'_j)]$, it is enough to prove $F(l'_{h'}) = c$. So take an even $j \geq i$. Then $l''_j = l_{j/2}$ and $l''_{j + 1} = l'_{j/2}$. Since $j \geq 2h$ and $j \geq 2h'$, we have $j/2 \geq h$ and $j/2 \geq h'$. Thus $F(l''_j) = F(l_{j/2}) = F(l_h) = c$ and $F(l''_{j + 1}) = F(l'_{j/2}) = F(l'_{h'})$. But $F(l''_j) = F(l''_{h''}) = F(l''_{j + 1})$ since $j \geq h''$. We conclude that $F(l'_{h'}) = c$.

  2.~Let $F \in \ASNIS$ and, by contradiction, let us assume that $A$ is an infinite set such that for all $c$ and $d$ there exists an $l \in \Seq$ such that $A_l \cap [d] =  A\cap [d] \wedge F(l) \neq c$. Notice that the latter formula is equivalent to a $\Sigma^0_0$ formula. Let $c$ be the number given by the previous point. We define a sequence $(l_m)$ by $l_m \defeq \min l \spc \big[l \in \Seq \wedge A_l \cap [m] = A \cap [m] \wedge F(l) \neq c\big]$ so that $\forall m \spc \big[A_{l_m} \cap [m] = A \cap [m] \wedge F(l_m) \neq c\big]$. Then $(l_m)$ weakly converges to $A$ while $\forall m \spc [F(l_m) \neq c]$, contradicting point 1.
\end{proof}

\section{Counterexample to $\IPP \leftrightarrow \FIPPone$}

\label{section-counterexample}
In this section we give a counterexample to $\FIPPone$. In particular, $\FIPPone$ is not equivalent to the true $\IPP$.

\begin{theorem}
  $\RCAzero$ refutes $\FIPPone$.
\end{theorem}

\begin{proof}
  We define $F \in AS$, $n$ and a sequence of colorings $f_k : [k] \to [n]$.

  We take $n \defeq 1$. Let us write $\dot{m}$ to mean that the number $m$ was given the color $0$ and $\ddot{m}$ to mean that it was given the color $1$.

  Let $\mathbb{O} \defeq \{1,3,5,\ldots\}$ be the set of the odd natural numbers and $\mathbb{E} \defeq \{0,2,4,\ldots\}$ be the set of the even natural numbers. Let us make the (non-standard) convention $\min \emptyset \defeq 0$. We define
  \begin{equation*}
    \begin{split}
      F : \mathbb{N} &\to \mathbb{N} \\
      l              &\mapsto
      \begin{cases}
        \min (A_l \cap \mathbb{O}) + \min (A_l \cap \mathbb{E}) + 2 & \text{if } l \in \Seq, \\
        0                                                           & \text{otherwise.}
      \end{cases}
    \end{split}
  \end{equation*}
  Let us prove $F \in AS$. Clearly $F$ is extensional. Consider a nested sequence with union $(l_m)$. Then we have a nested sequence $A_{l_0} \cap \mathbb{O} \subseteq A_{l_1} \cap \mathbb{O} \subseteq A_{l_2} \cap \mathbb{O} \subseteq \cdots$. So eventually the numbers $\min (A_{l_m} \cap \mathbb{O})$ will become constant. In an analogous way, eventually the numbers $\min (A_{l_m} \cap \mathbb{E})$ will become constant. So $F(l_m)$ will eventually become constant.

  We color each set $[k]$ in the following way:
  \begin{enumerate}
    \item the odd numbers are given the color $0$ and the even numbers are given the color $1$;
    \item except for the last two numbers $k - 1$ and $k$, where the odd number is given the color $1$ and the even number is given the color $0$.
  \end{enumerate}
  In the cases of $k = 0$ and $k = 1$, i.e., in the cases of the sets $[0]$ and $[1]$, we consider that $0$ and $1$ are the last two numbers, so we apply the second rule to them.

  Let us write the colored sets $[0],[1],[2],\ldots$ and, on the left of each set, the value of $F$ over the $0$- and $1$-color classes:
  \begin{equation*}
    \begin{array}{ccllllllllllllll}
      \dot{2}  & \ddot{2}  &&& \{ & \dot{0}   & \}        &           &           &           &           &          &          &        \\
      \dot{2}  & \ddot{3}  &&& \{ & \dot{0},  & \ddot{1}  & \}        &           &           &           &          &          &        \\
      \dot{4}  & \ddot{3}  &&& \{ & \ddot{0}, & \ddot{1}, & \dot{2}   & \}        &           &           &          &          &        \\
      \dot{5}  & \ddot{5}  &&& \{ & \ddot{0}, & \dot{1},  & \dot{2},  & \ddot{3}  & \}        &           &          &          &        \\
      \dot{7}  & \ddot{5}  &&& \{ & \ddot{0}, & \dot{1},  & \ddot{2}, & \ddot{3}, & \dot{4}   & \}        &          &          &        \\
      \dot{7}  & \ddot{7}  &&& \{ & \ddot{0}, & \dot{1},  & \ddot{2}, & \dot{3},  & \dot{4},  & \ddot{5}  & \}       &          &        \\
      \dot{9}  & \ddot{7}  &&& \{ & \ddot{0}, & \dot{1},  & \ddot{2}, & \dot{3},  & \ddot{4}, & \ddot{5}, & \dot{6}  & \}       &        \\
      \dot{9}  & \ddot{9}  &&& \{ & \ddot{0}, & \dot{1},  & \ddot{2}, & \dot{3},  & \ddot{4}, & \dot{5},  & \dot{6}, & \ddot{7} & \}     \\
      \vdots   & \vdots    &&&    & \vdots    & \vdots    & \vdots    & \vdots    & \vdots    & \vdots    & \vdots   & \vdots   & \ddots \\
    \end{array}
  \end{equation*}
  Notice that the cardinality of any $f_k$-color class is less than or equal to $|[k]| = k + 1$ which in turn is less than or equal to the value of $F$ over (a code for) that color class. So we have $\forall k \spc \forall l \in \Seq \spc \forall c \in [1] \spc [A_l = (f_k)^{-1}(c) \to |A_l| \leq F(l)]$, which falsifies $\FIPPone$.
\end{proof}

\section{Proofs of $\FIPPtwo \to \IPP$ and $\FIPPthree \to \IPP$}

In this section we give proofs in $\RCAzero$ of the implications $\FIPPtwo \to \IPP$ and $\FIPPthree \to \IPP$. Latter we study the reverse implications.

\begin{theorem} \label{forward} \mbox{}
  \begin{enumerate}
    \item $\RCAzero$ proves $\FIPPtwo \to \IPP$.
    \item $\RCAzero$ proves $\FIPPthree \to \IPP$.
  \end{enumerate}
\end{theorem}

\begin{proof}
  1.~$\FIPPtwo$ implies
  \begin{equation}
    \label{eq:weakFIPP}
    \forall n \spc \forall F \in AS \spc \forall f : \mathbb{N} \to [n] \spc \underbrace{{} \exists k \spc \exists l \in \Seq \spc \big[A_l \subseteq [k] \wedge |A_l| > F(l) \wedge f|_{A_l} \text{ constant}\big]}_{\equivdef B(F,f)}.
  \end{equation}
  Assume $\neg \IPP$. Then there exists $n$ and $f : \mathbb{N} \to [n]$ such that
  \begin{equation}
    \label{eq:AfiniteOrfNotConstant}
    \forall A \spc \exists m \spc (|A| \leq m \vee f|_A \text{ not constant}).
  \end{equation}
  If $A$ is given by a code $l$, then by point 2 of lemma \ref{lemma:cardinality} the formula ``$|A_l| \leq m \vee f|_{A_l} \text{ not constant}$'' is equivalent to some $\Sigma^0_0$ formula (using that the image of $f$ is in $[n]$), thus by primitive recursion we can define the function
  \begin{equation*}
    \begin{split}
      F : \mathbb{N} &\to \mathbb{N} \\
      l              &\mapsto
      \begin{cases}
        \min m \spc (|A_l| \leq m \vee f|_{A_l} \text{ not constant}) & \text{if } l \in \Seq, \\
        0                                                             & \text{otherwise.}
      \end{cases}
    \end{split}
  \end{equation*}

  Let us prove $F \in AS$. Take any nested sequence $(l_m)$ with union $A \defeq \bigcup_m A_{l_m}$. If $A$ is finite, then $\exists i \spc \forall j \geq i \spc (A_{l_i} = A_{l_j})$ (using $\BSigmaOne$ which is derivable in $\RCAzero$). Hence $F$ eventually stabilizes over $(l_m)$. If $A$ is infinite, then by (\ref{eq:AfiniteOrfNotConstant}) we have $\exists x,y \in A \spc [x \neq y \wedge f(x) \neq f(y)]$ and $\exists i \spc (x,y \in A_{l_i})$ which yields $\forall j \geq i \spc (x,y \in A_{l_j})$. Thus $\forall j \geq i \spc [F(l_i) = 0 = F(l_j)]$. This concludes the proof of $F \in AS$.

  So we have found $n$, $F \in AS$ and $f : \mathbb{N} \to [n]$ such that $\neg B(F,f)$, contradicting (\ref{eq:weakFIPP}).

  2.~The proof is analogous to the proof of point 1, except for the argument that $F \in \ASNIS$. Let us prove $F \in \ASNIS$. Take any sequence $(l_m)$ weakly convergent to an infinite set $A$. By (\ref{eq:AfiniteOrfNotConstant}) we get $\exists x,y \in A \spc [x \neq y \wedge f(x) \neq f(y)]$. Let $z \defeq \max(x,y)$. Since $(l_m)$ is weakly convergent to $A$ we have $\exists i \spc \forall j \geq i \spc (A_{l_j} \cap [z] = A \cap [z])$. But $x,y \in A \cap [z]$ and so $\forall j \geq i \spc (x,y \in A_{l_j})$, thus $\forall j \geq i \spc [F(l_i) = 0 = F(l_j)]$.
\end{proof}

Together with proposition \ref{Hirst} we get

\begin{corollary}
  $\WKLzero$ does not prove $\FIPPtwo$. Also $\WKLzero$ does not prove $\FIPPthree$.
\end{corollary}

\section{Continuous uniform boundedness}

In definition \ref{def:CUB} we will define a predicate $\cont(A)$ that, in particular, expresses the continuity of the functional
\begin{equation*}
  \phi : [n]^\mathbb{N} \to \mathbb{N}, \quad f \mapsto \min x \spc [A(f,x)].
\end{equation*}
Then we define a compactness principle $\CUB$ that roughly speaking expresses that if $\phi$ is continuous and total, then it is bounded on the compact $[n]^\mathbb{N}$. There is also a variant $\CUBprime$ that emphasizes that the conclusion only talks about an initial segment of $f$. However, it turns out that for the instances of $\CUB$ and $\CUBprime$ in which we are interested, the two principles are equivalent, as we show in proposition \ref{prop:CUBequivCUBprime}. In proposition \ref{prop:SigmaZeroEquivSigmaOne} we show that $\SigmaZeroCUB$ and $\SigmaZeroCUBprime$ can be upgraded to $\SigmaOneCUB$ and $\SigmaOneCUBprime$. In theorem \ref{theorem:CUBBigFive} we calibrate the strength of $\Phi$-$\CUB$ in terms of the ``big five'' subsystems of second order arithmetic.

\begin{definition}
  \label{def:CUB}
  The following definition is made within $\RCAzero$.
  \begin{enumerate}
    \item Let $A(f,\ul{x})$ be a formula with (among others) a distinguished set variable $f$ and a distinguished tuple of number variables $\ul{x}$. We say that $A$ is \emph{continuous} (w.r.t.\ $f,\ul{x}$), and write $\cont(A)$ (more precisely $\cont(A(f,\ul{x}))$), if and only if
      \begin{equation*}
        \forall f : \mathbb{N} \to [n] \spc \forall z \spc \exists y \spc \forall g : \mathbb{N} \to [n] \spc \big[\bar f y = \bar g y \to \forall \ul{x} \leq z \big(A(f,\ul{x}) \leftrightarrow A(g,\ul{x})\big)\big],
      \end{equation*}
      where the variable $n$ doesn't occur free in $A$.
  \item The \emph{continuous uniform boundedness principle} $\CUB$ is the schema
    \begin{equation*}
      \forall n \spc \big[\big(\cont(A) \wedge \forall f : \mathbb{N} \to [n] \spc \exists \ul{x} \spc A(f,\ul{x})\big) \to \exists z \spc \forall f : \mathbb{N} \to [n] \spc \exists \ul{x} \leq z \spc A(f,\ul{x})\big].
    \end{equation*}
    We denote by $\Gamma$-$\CUB$ the restriction of $\CUB$ to formulas $A(f,\ul{x})$ in $\Gamma$.
  \item The \emph{variant continuous uniform boundedness principle} $\CUBprime$ is the schema
    \begin{equation*}
      \begin{split}
       \forall n \spc \big[\big(\cont(A) \wedge \forall f : \mathbb{N} \to [n] \spc \exists \ul{x} \spc A(f,\ul{x})\big) \to & \\
        \exists z \spc \forall f : \mathbb{N} \to [n] \spc \exists \ul{x} \leq z \spc \big( A(f,\ul{x}) \wedge \forall g &: \mathbb{N} \to [n] \spc \big(\bar f z = \bar g z \to A(g,\ul{x})\big)\big)\big].
      \end{split}
    \end{equation*}
    We denote by $\Gamma$-$\CUBprime$ the restriction of $\CUBprime$ to formulas $A(f,\ul{x})$ in $\Gamma$.
  \end{enumerate}
\end{definition}

\begin{proposition}
  \label{prop:SigmaZeroEquivSigmaOne}
  $\RCAzero$ proves the following.
  \begin{enumerate}
    \item $\SigmaZeroCUB \leftrightarrow \SigmaOneCUB$.
    \item $\SigmaZeroCUBprime \leftrightarrow \SigmaOneCUBprime$.
  \end{enumerate}
\end{proposition}

\begin{proof}
  1.~The right-to-left implication is trivial. Let us consider the left-to-right implication. Consider any $\Sigma^0_1$ formula $\exists \ul{w} \spc A$ where $A$ is a $\Sigma^0_0$ formula. We assume the part $\forall f : \mathbb{N} \to [n] \spc \exists \ul{x} \spc [\exists \ul{w} \spc A(f,\ul{x},\ul{w})]$ of the assumption of $\SigmaOneCUB$. By point 1 of lemma \ref{lemma:uniformContinuity} we have $\cont(A)$ (w.r.t.\ $f,\ul{x},\ul{w}$). By $\SigmaZeroCUB$ applied to $A$ we get $\exists z \spc \forall f : \mathbb{N} \to [n] \spc \exists \ul{x},\ul{w} \leq z \spc A(f,\ul{x},\ul{w})$. From here we get the conclusion $\exists z \spc \forall f : \mathbb{N} \to [n] \spc \exists \ul{x} \leq z \spc [\exists \ul{w} \spc A(f,\ul{x},\ul{w})]$ of $\SigmaOneCUB$.

  2.~The proof is analogous to the proof of the previous point.
\end{proof}

\begin{remark}
  By point 1 of lemma \ref{lemma:uniformContinuity}, $\cont(A)$ is always satisfied for $\Sigma^0_0$ formulas $A$ and so can be dropped in $\SigmaZeroCUB$. The proof of point 1 in proposition \ref{prop:SigmaZeroEquivSigmaOne} above shows that dropping $\cont(A)$ also in $\SigmaOneCUB$ results in an equivalent principle. For $\PiCUB$ this is no longer the case (see the comments at the end of this paper).
\end{remark}

\begin{proposition}
  \label{prop:CUBequivCUBprime}
  $\RCAzero$ proves the following.
  \begin{enumerate}
    \item $\SigmaZeroCUB \leftrightarrow \SigmaZeroCUBprime$.
    \item $\PiCUB \leftrightarrow \PiCUBprime$.
    \item $\CUB \leftrightarrow \CUBprime$.
  \end{enumerate}
\end{proposition}

\begin{proof}
  1.~The right-to-left implication is trivial. Let us prove the left-to-right implication. We assume the premise $\cont(A) \wedge \forall f : \mathbb{N} \to [n] \spc \exists \ul{x} \spc A(f,\ul{x})$ of $\SigmaZeroCUBprime$, where $A(f,\ul{x})$ is a $\Sigma^0_0$ formula. Then by $\SigmaZeroCUB$ we have
  \begin{equation}
    \label{eq:CUBequivCUBprime1}
    \exists z' \spc \forall f : \mathbb{N} \to [n] \spc \exists \ul{x} \leq z' \spc A(f,\ul{x},y).
  \end{equation}
  By point 1 of lemma \ref{lemma:uniformContinuity} we have
  \begin{equation}
    \label{eq:CUBequivCUBprime2}
    \exists y \spc \forall f,g : \mathbb{N} \to [n] \spc \forall \ul{x} \leq z' \spc \big[\bar f y = \bar g y \to \big(A(f,\ul{x}) \leftrightarrow A(g,\ul{x})\big)\big].
  \end{equation}
  Let $z \defeq \max(y,z')$. Then from (\ref{eq:CUBequivCUBprime1}) and (\ref{eq:CUBequivCUBprime2}) we get the conclusion of $\SigmaZeroCUBprime$.

  2.~The right-to-left implication is trivial. Let us prove the left-to-right implication. We assume the premise $\cont(\forall \ul{w} \spc A(f,\ul{x},\ul{w})) \wedge \forall f : \mathbb{N} \to [n] \spc \exists \ul{x} \spc [\forall \ul{w} \spc A(f,\ul{x},\ul{w})]$ of $\PiCUBprime$, where $\forall \ul{w} \spc A(f,\ul{x},\ul{w})$ is a $\Pi^0_1$ formula and $A(f,\ul{x},\ul{w})$ is a $\Sigma^0_0$ formula. Then
  \begin{equation*}
    \forall f : \mathbb{N} \to [n] \spc \exists \ul{x} \spc \exists y \spc \underbrace{\big[\forall \ul{w} \spc A(f,\ul{x},\ul{w}) \wedge \forall g : \mathbb{N} \to [n] \spc \big(\bar f y = \bar g y \to \forall \ul{w} \spc A(g,\ul{x},\ul{w})\big)\big]}_{{} \equivdef B}.
  \end{equation*}
  Note that $\bar f y = \bar g y$ is equivalent to
  \begin{equation*}
    \forall i < y \spc \forall z\le n \spc [(i,z)\in f\leftrightarrow (i,z) \in g] \in \Sigma^0_0.
  \end{equation*}
  Moving the quantifiers $\forall \ul{w}$ and $\forall g : \mathbb{N} \to [n]$ in $B$ to the front of $B$ we get an equivalent formula of the form $\forall \ul{w} \spc \forall g : \mathbb{N} \to [n] \spc C$ were $C$ is a $\Sigma^0_0$ formula. By point 2 of lemma \ref{lemma:uniformContinuity}, $\forall g : \mathbb{N} \to [n] \spc C$ is equivalent to a $\Pi^0_1$ formula, so $B$ is equivalent to a $\Pi^0_1$ formula. Therefore we can apply $\PiCUB$ to $B$ (note that $\cont(B)$ w.r.t.\ $f,\ul{x},y$ since $\cont(\forall\ul{w} \spc A(f,\ul{x},\ul{w}))$) getting
  \begin{equation*}
    \exists z \spc \forall f : \mathbb{N} \to [n] \spc \exists \ul{x},y \leq z \spc \big[\forall \ul{w} \spc A(f,\ul{x},\ul{w}) \wedge \forall g : \mathbb{N} \to [n] \spc \big(\bar f y = \bar g y \to \forall \ul{w} \spc A(g,\ul{x},\ul{w})\big)\big].
  \end{equation*}
  Now replacing $y$ by $z$ in $\bar f y = \bar g y$ we get the conclusion of $\PiCUBprime$.

  3.~The proof is analogous to the proof of point 2, disregarding the considerations about the complexity of $B$.
\end{proof}

\begin{theorem} \mbox{}
  \label{theorem:CUBBigFive}
  \begin{enumerate}
    \item $\RCAzero$ proves $\SigmaZeroCUB \leftrightarrow \WKLzero$.
    \item $\RCAzero$ proves $\PiCUB \leftrightarrow \ACAzero$.
    \item $\RCA$ (not $\RCAzero$) proves $\CUB \leftrightarrow \Ztwo$.
  \end{enumerate}
\end{theorem}

\begin{proof}
  1.~`$\rightarrow$':~We assume $\SigmaZeroCUB$ and, by contradiction, $\neg \WKLzero$. Then we have an infinite binary tree $T \subseteq 2^{<\mathbb{N}}$ with no infinite path, i.e., $\forall f : \mathbb{N} \to [1] \spc \exists x \spc (\bar fx \notin T)$ where the formula $\bar fx \notin T$ is $\Delta^0_1$ and hence $\Sigma^0_1$. By $\SigmaOneCUB$ and so (using proposition \ref{prop:SigmaZeroEquivSigmaOne}) also by $\SigmaZeroCUB$ we have $\exists z \spc \forall f : \mathbb{N} \to [1] \spc \exists x \leq z \spc (\bar f x \notin T)$. This means that every branch in $T$ has length bounded by $z - 1$, so the binary tree $T$ is finite, contradicting the fact that it is infinite.

  `$\leftarrow$':~First we show that $\RCAzero$ proves
  \begin{equation*}
    \WKLzero \to \big[\big(\forall k \spc \exists f : \mathbb{N} \to [n] \spc \forall \ul{x} \leq k \spc A(f,\ul{x})\big) \to \big(\exists f : \mathbb{N} \to [n] \spc \forall \ul{x} \spc A(f,\ul{x})\big)\big],
  \end{equation*}
  where $A$ is $\Sigma^0_0$. We assume $\WKLzero$ and $\forall k \spc \exists f : \mathbb{N} \to [n] \spc \forall \ul{x} \leq k \spc A(f,\ul{x})$. By point 3 of lemma \ref{lemma:uniformContinuity} we can write $A(f,\ul{x})$ as $\forall m \spc B(\bar f m,\ul{x})$ where $B$ is $\Sigma^0_0$. By $\Sigma^0_0$ comprehension we define the bounded tree $T \defeq \{\tau \in [n]^{< \mathbb{N}} : \forall \ul{x},m \leq \lh(\tau) \spc B(\bar \tau m,\ul{x})\}$. We have $\forall p \spc \exists \tau \in T \spc [\lh(\tau) = p]$: taking $k = p$ in our assumption we get an $f : \mathbb{N} \to [n]$ such that $\forall \ul{x},m \leq p \spc B(\bar f m,\ul{x})$ where $\bar f m = \bar \tau m$ for $\tau \defeq \bar f p \in T$ with length $\lh(\tau) = p$. So $T$ is infinite, thus by $\WKLzero$ (actually by bounded K\"onig's lemma that is equivalent to $\WKLzero$ over $\RCAzero$ as proved in lemma IV.1.4 in \cite{Simpson}) there is an infinite path $f : \mathbb{N} \to [n]$ through $T$. Then $\forall \ul{x} \spc A(f,\ul{x})$, i.e., $\forall \ul{x},m \spc B(\bar f m,\ul{x})$: for $p \defeq \max(\ul{x},m)$ we have $\tau \defeq \bar f p \in T$, i.e., $\forall \ul{x'},m' \leq p \spc B(\bar \tau m',\ul{x'})$ where $\bar \tau m' = \bar f m'$, and so taking $\ul{x'} = \ul{x}$ and $m' = m$ we get $B(\bar fm,\ul{x})$.

  Finally, we show that $\RCAzero$ proves $\WKLzero \to \SigmaZeroCUB$ taking the contrapositive of the inner implication proved in the previous paragraph.

  2.~It is enough to show that $\RCAzero$ proves that $\PiCUB$ is equivalent to $\Pi^0_1$ comprehension, since $\Pi^0_1$ comprehension is equivalent to $\Sigma^0_1$ comprehension and in turn, as proved in lemma III.1.3 in \cite{Simpson}, $\Sigma^0_1$ comprehension is equivalent over $\RCAzero$ to $\ACAzero$.

  `$\rightarrow$':~Let us prove that $\PiCUB$ implies $\Pi^0_1$ comprehension. Consider any $\Pi^0_1$ formula $\forall \ul{m} \spc A(x,\ul{m})$ where $A(x,\ul{m})$ is a $\Sigma^0_0$ formula. By contradiction, we assume $\neg \exists X \spc \forall x \spc [x \in X \leftrightarrow \forall \ul{m} \spc A(x,\ul{m})]$, i.e., $\neg \exists f : \mathbb{N} \to [1] \spc \forall x \spc [f(x) = 0 \leftrightarrow \forall \ul{m} \spc A(x,\ul{m})]$. Then
  \begin{equation*}
    \forall f : \mathbb{N} \to [1] \spc \exists x,\ul{m} \spc \underbrace{\forall \ul{m'} \spc \neg \big[\big(f(x) = 0 \to A(x,\ul{m})\big) \wedge \big(A(x,\ul{m'}) \to f(x) = 0\big)\big]}_{{} \equivdef B(f,x,\ul{m})}.
  \end{equation*}
  By $\PiCUB$ applied to the (continuous) $\Pi^0_1$ formula $B$ we get a $z$ such that
  \begin{equation*}
    \forall f : \mathbb{N} \to [1] \spc \exists x \leq z \spc \neg \big[\big(f(x) = 0 \to \forall \ul{m} \leq z \spc A(x,\ul{m})\big) \wedge \big(\forall \ul{m'} \spc A(x,\ul{m'}) \to f(x) = 0\big)\big].
  \end{equation*}
  But that is contradicted by the function
  \begin{equation*}
    \begin{split}
      f : \mathbb{N} &\to [1] \\
      x              &\mapsto
        \begin{cases}
          0 & \text{if } \forall \ul{m} \leq z \spc A(x,\ul{m}), \\
          1 & \text{otherwise}
        \end{cases}
    \end{split}
  \end{equation*}
  which is definable by $\Sigma^0_0$ comprehension.

  `$\leftarrow$':~Now let us see that $\Pi^0_1$ comprehension implies $\PiCUB$. We assume $\Pi^0_1$ comprehension, and, therefore, we have $\ACAzero$. We assume the premise $\cont(A) \wedge {\forall f : \mathbb{N} \to [n]} \spc \exists \ul{x} \spc A(f,\ul{x})$ of $\PiCUB$, where $A(f,\ul{x})$ is a $\Pi^0_1$ formula, and we want to prove the conclusion $\exists z \spc \forall f : \mathbb{N} \to [n] \spc \exists \ul{x} \leq z \spc A(f,\ul{x})$ of $\PiCUB$. Take any $f : \mathbb{N} \to [n]$. In $\ACAzero$ there exists $\min z \spc [\exists \ul{x} \leq z \spc A(f,\ul{x})]$. Consider the functional
  \begin{equation*}
    \begin{split}
      \phi : [n]^\mathbb{N} &\to \mathbb{N} \\
      f                     &\mapsto \min z \spc [\exists \ul{x} \leq z \spc A(f,\ul{x})].
    \end{split}
  \end{equation*}
  This functional cannot directly be formed in $\ACAzero$ as it is a 3rd order object. However, we will show now that it has a (2nd order) code as a continuous function in the sense of \cite{Simpson}. Let
  \begin{align*}
    \l \in \Seq_{\leq n} \defequiv {} & l \in \Seq \wedge l \concatenation o : \mathbb{N} \to [n], \\
    B(l)                 \defequiv {} & l \in \Seq_{\leq n} \wedge \forall l' \in \Seq_{\leq n} \\
                                      & \big[l \subseteq l' \to \min z \spc [\exists \ul{x} \leq z \spc A(l \concatenation o,\ul{x})] = \min z \spc [\exists \ul{x} \leq z \spc A(l' \concatenation o,\ul{x})]\big].
  \end{align*}
  In $\ACAzero$ there exists the following function $\alpha$ which -- as we will argue now -- is an associate for $\phi$:
  \begin{equation*}
    \alpha(l) \defeq
    \begin{cases}
      \min z \spc [\exists \ul{x} \leq z \spc A(l \concatenation o,\ul{x})] + 1 & \text{if } B(l), \\
      0 & \text{otherwise.}
    \end{cases}
  \end{equation*}
  Take any $\beta : \mathbb{N} \to [n]$. By $\cont(A)$ there exists a $y$ such that
  \begin{equation*}
    \forall g : \mathbb{N} \to [n] \spc \big[\bar \beta y = \bar g y \to \forall \ul{x} \leq \min z \spc [\exists \ul{\tilde{x}} \leq z \spc 
A(\beta,\ul{\tilde{x}})] \spc \big(A(\beta,\ul{x}) \leftrightarrow A(g,\ul{x})\big)\big],
  \end{equation*}
  thus
  \begin{equation}
    \label{eq:continuityPointForBarBetaY}
    \forall g : \mathbb{N} \to [n] \spc \big[\bar \beta y = \bar g y \to \min z \spc [\exists \ul{x} \leq z \spc A(\beta,\ul{x})] = \min z \spc [\exists \ul{x} \leq z \spc A(g,\ul{x})]\big].
  \end{equation}

  (a)~First we prove that there exists an $m$ such that $\alpha(\bar \beta m) > 0$. Let $m \defeq y$. We have $B(\bar \beta m)$: for all $l' \in \Seq_{\leq n}$ such that $\bar \beta m \subseteq l'$, taking $g = \bar \beta m \concatenation o$ and $g = l' \concatenation o$ in (\ref{eq:continuityPointForBarBetaY}) we get, respectively,
  \begin{align*}
    \min z \spc [\exists \ul{x} \leq z \spc A(\beta,\ul{x})] &= \min z \spc [\exists \ul{x} \leq z \spc A(\bar \beta m \concatenation o,\ul{x})], \\
    \min z \spc [\exists \ul{x} \leq z \spc A(\beta,\ul{x})] &= \min z \spc[\exists \ul{x} \leq z \spc A(l' \concatenation o,\ul{x})].
  \end{align*}
  Thus $\min z \spc [\exists \ul{x} \leq z \spc A(\bar \beta m \concatenation o,\ul{x})] = \min z \spc [\exists \ul{x} \leq z \spc A(l' \concatenation o,\ul{x})]$. Since we have $B(\bar \beta m)$, then by definition of $\alpha$ we have $\alpha(\bar \beta m) > 0$.

  (b)~Now we take the least $m$ such that $\alpha(\bar \beta m) > 0$ and we prove $\alpha(\bar \beta m) = \phi(\beta) + 1$. Since $\alpha(\bar \beta m) > 0$ we have $B(\bar \beta m)$. Let $w \defeq \max(m,y)$. By $B(\bar \beta m)$ and taking $g = \bar \beta w \concatenation o$ in (\ref{eq:continuityPointForBarBetaY}) we get, respectively,
  \begin{align*}
    \overbrace{\min z \spc [\exists \ul{x} \leq z \spc A(\bar \beta m \concatenation o,\ul{x})]}^{{} = \alpha(\bar \beta m) - 1} &=
\min z \spc [\exists \ul{x} \leq z \spc A(\bar \beta w \concatenation o,\ul{x})], \\
    \underbrace{\min z \spc [\exists \ul{x} \leq z \spc A(\beta,\ul{x})]}_{{} = \phi(\beta)} &= \min z \spc [\exists \ul{x} \leq z \spc A(\bar \beta w \concatenation o,\ul{x})].
  \end{align*}
  Thus $\alpha(\bar \beta m) = \phi(\beta) + 1$.

  This concludes the proof that $\alpha$ is an associate for $\phi$. Thus by lemma \ref{lemma:associateImpliesHavingCode}, $\phi$ has a code as continuous function. Since $[n]^{\mathbb{N}}$ is (provably already in $\RCAzero$) a compact metric space (see \cite{Simpson} (examples III.2.6)) it follows from \cite{Brown} (Theorem 4.1) that (provably in $\WKLzero$ and so a-fortiori in $\ACAzero$) $\phi$ has an upper bound $z$. Then $\forall f : \mathbb{N} \to [n] \spc \exists \ul{x} \leq \phi(f) \leq z \spc A(f,\ul{x})$.

  3.~We prove that $\Ztwo$ implies $\CUB$ essentially in the same way that we proved in the previous point that $\Pi^0_1$ comprehension implies $\PiCUB$. To see that, conversely, $\CUB$ implies (relative to $\RCA$) $\Ztwo$ it is enough to show that $\CUB$ implies the comprehension axiom for arbitrary formulas $A$. By induction on $z$ we prove $\forall z \spc \exists m \in \Seq_{\leq 1} \spc \big[\lh m = z + 1 \wedge \forall x \leq z \spc \big(m(x) = 0 \leftrightarrow A(x)\big)\big]$. By contradiction assume $\neg \exists X \spc \forall x \spc [x \in X \leftrightarrow A(x)]$, that is, $\forall f : \mathbb{N} \to [1] \spc \exists x \spc B(f,x)$ where $B(f,x) \defequiv \neg [f(x) = 0 \leftrightarrow A(x)]$. Clearly $\cont(B)$, so applying $\CUB$ to $B$ we get $\exists z \spc \forall f: \mathbb{N} \to [1] \spc \exists x \leq z \spc B(f,x)$. But this is contradicted by $f = m \concatenation o$.
\end{proof}

\begin{remark}
  As the proof above shows, the strength of the various CUB-principles considered does not depend on whether they are formulated with general $n$ or just with $n = 1$ (this can also be seen directly using the construction on page 220 in \cite{TroelstraDalen(88)}). Note, however, that $\IPP$ restricted to $n = 1$ or any fixed $n$ is much weaker (and essentially provable in pure logic) than $\IPP$.
\end{remark}

\section{Proofs of $\IPP \to \FIPPtwo$ and  $\IPP \to \FIPPthree$ using continuous uniform boundedness}

In the previous section we calibrated the strength of $\SigmaZeroCUB$ and $\PiCUB$ in terms of the ``big five''. In this section we give upper bounds on the strength of the implications $\IPP \to \FIPPtwo$ and $\IPP \to \FIPPthree$.

\begin{theorem} \mbox{}
  \label{PiCUB-application}
  \begin{enumerate}
    \item $\RCAzero + \SigmaZeroCUB$ proves $\IPP \to \FIPPtwo$.
    \item $\RCAzero + \PiCUB$ proves $\IPP \to \FIPPthree$.
  \end{enumerate}
\end{theorem}

\begin{proof}
  1.~Take any $n$ and $F \in AS$. Let
  \begin{equation*}
    B(f,k) \defequiv \exists l \in \Seq \spc \big[A_l \subseteq [k] \wedge |A_l| > F(l) \wedge f|_{A_l} \text{ constant}\big].
  \end{equation*}
  Let us prove $\forall f : \mathbb{N} \to [n] \spc \exists k \spc B(f,k)$. Take any $f : \mathbb{N} \to [n]$. By $\IPP$ there exists an infinite color class $f^{-1}(c)$ with $c \in [n]$. By primitive recursion in $f$ define a sequence $(l_m)$ where each $l_m \in \Seq$ is such that $A_{l_m} = f^{-1}(c) \cap [m]$. Then $(l_m)$ is a nested sequence with union $f^{-1}(c)$, so there exists a $k'$ such that $\forall m \geq k' \spc [F(l_m) = F(l_{k'})]$. Since $f^{-1}(c)$ is infinite, by points 3 and 4 of lemma \ref{lemma:cardinality} there exists a $k''$ such that $\forall m \geq k'' \spc [|A_{l_m}| > F(l_{k'})]$. Let $k \defeq \max(k',k'')$. Then we have $B(f,k)$. This finishes the proof of $\forall f : \mathbb{N} \to [n] \spc \exists k \spc B(f,k)$.

  Notice that $A_l \subseteq [k]$ and ``$f|_{A_l} \text{ constant}$'' are equivalent to some bounded formulas and $|A_l| > F(l)$ is equivalent to the $\Sigma^0_0$ formula
  \begin{equation*}
    \exists j\le l \spc \exists i<j \spc [|A_l|=j\wedge (l,i)\in F].
  \end{equation*}
  Thus $B(f,k)$ is equivalent to a $\Sigma^0_1$ formula. Also notice that we have $\cont(B)$.

  By $\SigmaOneCUB$ we get $\exists k \spc \forall f : \mathbb{N} \to [n] \spc \exists k''' \leq k \spc B(f,k''')$. Since $B(f,k)$ is monotone in $k$, i.e., $k''' \leq k \wedge B(f,k''') \to B(f,k)$ we get $\exists k \spc \forall f : \mathbb{N} \to [n] \spc B(f,k)$. Now since $f$ is only applied in $B(f,k)$ to arguments in $[k]$, we can consider only functions $f$ with domain $[k]$: $\exists k \spc \forall f : [k] \to [n] \spc B(f,k)$.

  2.~Take any $n$ and $F \in \ASNIS$. Let
  \begin{equation*}
    B(f,k) \defequiv \exists c \in [n] \spc \forall l \in \Seq \spc \big[A_l \cap [k] = f^{-1}(c) \cap [k] \to |A_l| > F(l)\big].
  \end{equation*}
  Let us prove $\forall f : \mathbb{N} \to [n] \spc \exists k \spc B(f,k)$. Take any $f : \mathbb{N} \to [n]$. As in the proof of the previous point we have an infinite color class $f^{-1}(c)$ and a sequence $(l_m)$ where each $l_m \in \Seq$ is such that $A_{l_m} = f^{-1}(c) \cap [m]$. Then $(l_m)$ weakly converges to $f^{-1}(c)$, so there exists a $k'$ such that
  \begin{equation}
    \label{eq:kOne}
    \forall m \geq k' \spc [F(l_m) = F(l_{k'})].
  \end{equation}
  Since $f^{-1}(c)$ is infinite, then by points 3 and 4 of lemma \ref{lemma:cardinality} there exists a $k''$ such that
  \begin{equation}
    \label{eq:kTwo}
    \forall m \geq k'' \spc [|f^{-1}(c) \cap [m]| > F(l_{k'})].
  \end{equation}
  By lemma \ref{lemma:limitUniquenessAndStabilityPoint}.2 there exist $c'$ and $k'''$ such that
  \begin{equation}
    \label{eq:kThree}
    \forall l \in \Seq \spc \big[A_l \cap [k'''] = f^{-1}(c) \cap [k'''] \to F(l) = c'\big].
  \end{equation}
  Let $k \defeq \max(k',k'',k''')$. Taking $m = k$ in (\ref{eq:kOne}) and $l = l_k$ in (\ref{eq:kThree}) we get $c' = F(l_{k'})$. For all $l \in \Seq$, if $A_l \cap [k] = f^{-1}(c) \cap [k]$, then $A_l \cap [k'''] = f^{-1}(c) \cap [k''']$ and by point 4 of lemma \ref{lemma:cardinality} we have $|A_l| \geq |f^{-1}(c) \cap [k]|$, thus by (\ref{eq:kTwo}) and (\ref{eq:kThree}) we get $|A_l| > F(l)$. This finishes the proof of $\forall f : \mathbb{N} \to [n] \spc \exists k \spc B(f,k)$.

  Notice that $A_l \cap [k] = f^{-1}(c) \cap [k]$ is equivalent to a $\Sigma^0_0$ formula and by point 2 of lemma \ref{lemma:cardinality} $|A_l| > F(l)$ is (as shown above) equivalent to a $\Sigma^0_0$ formula, so (using $B\Sigma^0_1$) $B$ is equivalent to a $\Pi^0_1$ formula. Also notice that we have $\cont(B)$ since the only occurrence of $f$ in $B$ is $f^{-1}(c) \cap [k]$, i.e., $(f|_{[k]})^{-1}(c)$.

  Now, analogously to the proof of the previous point, apply $\PiCUB$, use the monotonicity of $B(f,k)$ on $k$, and notice that we can restrict the functions $f$ to $[k]$. Finally, taking $l = l_k$ (so $A_{l_k} = f^{-1}(c)$ where $f : [k] \to [n]$) we get the $\FIPPthree$.
\end{proof}

\begin{corollary} \mbox{}
  \label{equivalence}
  \begin{enumerate}
    \item $\WKLzero$ proves $\IPP \leftrightarrow \FIPPtwo$.
    \item $\ACAzero$ proves $\IPP \leftrightarrow \FIPPthree$.
  \end{enumerate}
\end{corollary}

While the first equivalence shows that $\FIPPtwo$ is a nontrivial finitization of $\IPP$ as neither principle is derivable in $\WKLzero$, the second equivalence does not establish this for $\FIPPthree$ since $\ACAzero$ not only proves $\IPP$ (and hence $\FIPPthree$) but even much stronger principles (e.g., Ramsey's theorem $\RT(k)$ for every fixed $k$, see \cite{Simpson}, or, on the arithmetical side, $\BSigmaInfty$). So while the fact that $\RCAzero$ suffices to prove $\FIPPthree\to \IPP$ shows that $\FIPPthree$
is strong enough to count as a ``finitization'' of $\IPP$, the fact that for the other direction we only have proofs using $\ACAzero$ leaves open the possibility that $\FIPPthree$ may be too strong to be a faithful finitization of $\IPP$.

\section{Historical comments on $\CUB$}

Without the continuity assumption $\cont(A)$, principles of the form $\CUB$ feature prominently in intuitionistic mathematics under the label of ``fan principles''. In fact, in intuitionistic analysis it is common to assume (classically inconsistent) continuity principles that, in particular, imply $\cont(A)$ (see \cite{TroelstraDalen(88)}). In our language context of 2nd order arithmetic, $\Sigma^0_0$ formulas $A$ automatically satisfy $\cont(A)$ and so in $\SigmaZeroCUB$ and (by its reduction to $\SigmaZeroCUB$) even in $\SigmaOneCUB$ one can drop the assumption $\cont(A)$. However, in contexts formulated in the language in all finite types over $\mathbb{N}$, the corresponding version without $\cont(A)$, called $\SigmaOneUB$, is not valid in the full type structure over $\mathbb{N}$ but satisfies very useful conservation results. $\SigmaOneUB$ was first introduced in \cite{Kohlenbach(95A)} and is studied in detail in \cite{KohlenbachApplied} (for a systematic proof-theoretic treatment of even more general forms of uniform boundedness by a specially designed so-called bounded functional interpretation see \cite{FerreiraOliva}). Recently in \cite{Kohlenbach(06),KohlenbachApplied}, $\SigmaOneUB$ was generalized to a principle $\ExistsUBX$ dealing with uniformities in the absence of compactness for abstract bounded metric and hyperbolic spaces. Again, while not valid in the intended model, the principle satisfies strong conservation theorems and so can be used safely for proofs of large classes of statements.

With classical logic alone (essentially), however, even $\PiCUB$ becomes inconsistent (and so in particular over $\RCAzero$ and much weaker systems) if the assumption $\cont(A)$ is dropped. E.g., just consider the logically valid statement
\begin{equation*}
  \forall f : \mathbb{N} \to [1] \spc \exists x \in \mathbb{N} \spc \underbrace{\forall y \in \mathbb{N} \spc [f(y)=0 \to f(x)=0]}_{{} \equivdef A(f,x) \in \Pi^0_1}.
\end{equation*}
Then $\CUB$ without the continuity assumption $\cont(A)$ (which does not hold here) would imply that
\begin{equation*}
  \exists z \spc \forall f : \mathbb{N} \to [1] \spc \big[\exists y \in \mathbb{N} \spc \big(f(y) = 0\big) \to \exists x \le z \spc \big(f(x) = 0\big)\big]
\end{equation*}
which obviously is wrong. A syntactic condition that guarantees $\cont(A)$ to hold is that $A(f,x)$ can be written as $\tilde{A}\big(\bar{f}(t(x)),x\big)$ for some number term $t$ (possibly with further number parameters of $A$), where $\tilde{A}(z,x)\in\Pi^0_1$ does not contain $f$. This is the case in the use of $\PiCUB$ in the proof of theorem \ref{PiCUB-application}.2 (with $t(k,c) \defeq k + 1$). In fact, (a version of) such a form (denoted by $\PiOneUBRes$) of $\CUB$ is considered in \cite{Kohlenbach(97A)}, where it is shown to imply the Bolzano-Weierstra\ss{} property of $[0,1]^d$ (over an extremely weak base system). Moreover, the proof of theorem \ref{theorem:CUBBigFive}.2 immediately shows that $\PiOneUBRes$ still implies $\Pi^0_1$ comprehension.

\end{document}